\documentclass[12pt]{article}
\usepackage{mathrsfs}
\usepackage{graphicx,amsmath,amsfonts,amssymb,amscd,amsthm}
\newtheoremstyle{kai}
{3pt} {3pt} {} {} {\bfseries} {.} {.5em} {}
\def\EquationsBySection{\def\theequation
{\thesection.\arabic{equation}}%
\@addtoreset{equation}{section}}

\newcommand\old[1]{}
\newcommand{\pend}{\hfill \thicklines \framebox(6.6,6.6)[l]{}}

\renewenvironment{proof}{\noindent {\it  Proof.} \rm}{\pend}
\newtheorem{theorem}{Theorem}[section]
\newtheorem{lemma}{Lemma}[section]

\newtheorem{remark}{Remark}[section]

\newtheorem{example}{Example}[section]

\EquationsBySection \makeatother
\usepackage{indentfirst}

\begin{document}
\pagestyle{plain}
\title
{\bf Comparison Theorem for  Stochastic Differential
Delay Equations with Jumps}
\author{Jianhai Bao$^{a,b}$ and  Chenggui Yuan$^{b,}$\thanks{ E-mail address: C.Yuan@swansea.ac.uk} \\
\vspace{0.2cm} \small{\it $^a$School of Mathematics, Central South
University, Changsha, Hunan 410075, PR China
}\\
\vspace{0.2cm} \small{\it $^b$Department of Mathematics, University
of
Wales Swansea, Swansea SA2 8PP, UK}\\
}

\date{}
\maketitle

\begin{abstract}{\rm
In this paper we  establish a comparison theorem for
 stochastic differential {\it delay } equations with jumps. An
example is constructed to demonstrate that the comparison theorem
need not hold whenever the diffusion term contains a delay function
although the jump-diffusion coefficient could contain a delay
function. Moreover, another example is established to show that the
comparison theorem is not necessary to be true provided that the
jump-diffusion term is non-increasing with respect to the delay
variable.
}\\

\end{abstract}
\noindent

\section{Introduction}

For most of the practical cases, the dynamical systems will be
disturbed by some stochastic perturbation \cite{st05}. One type of
stochastic perturbation is continuous and can be modeled by
stochastic integral with respect to the continuous martingale, e.g.,
Brownian motion. Non-Gaussian random processes also play an important
role in modelling stochastic dynamical systems (see, for example,
Applebaum \cite{a04}, Situ \cite{st05}, Peszat and Zabczyk
\cite{pz06}). Typical examples of non-Gaussian stochastic processes
are L\'{e}vy processes and processes arising by Poisson random
measures. In \cite{w01}, Woyczy\'{n}ski describes a number of
phenomena from fluid mechanics, solid state physics, polymer
chemistry, economic science, etc., for which non-Gaussian L\'{e}vy
processes can be used as their mathematical model in describing the
related probability behaviour. On the other hand, control
engineering intuition suggests that time-delays are common in
practical systems and are often the cause of instability and/or poor
performance \cite{ymy08}. Moreover, it is usually difficult to
obtain accurate values for the delay and conservative estimates
often have to be used. The importance of time delay has already
motivated by several studies on the stability of stochastic
diffusion with time delay (e.g., \cite{csl98} and \cite{99p}).

In the past few years, comparison theorems for two stochastic
differential equations (SDEs) have received a lot of attention,
for example, Anderson \cite{a72}, Gal'cuk and Davis \cite{gal}, Ikeda and Watanable \cite{iw77}, Mao
\cite{m91}, O'Brien \cite{obr},  Yamada \cite{y73}, Yan \cite{yan}  and references therein.
Recently, the comparison theorem has made a great development  that
Peng and Zhu \cite{pz06} obtain a necessary and
sufficient condition for comparison theorem of SDEs with jumps by applying a
criteria of ``viability condition",
 Peng and Yang \cite{py}
give a comparison theorem for anticipated backward stochastic
differential equations, and  for a class of  SDEs with delay, Yang, Mao and Yuan
\cite{ymy08} also establish a  comparison theorem.

In this paper we shall establish a comparison theorem for
 stochastic differential {\it delay } equations
(SDDEs) with jumps. It should be pointed out that the approach of this paper
is inspired by Peng and Yang \cite{py}, Peng and Zhu \cite{pz06} and
Yang, Mao and Yuan \cite{ymy08}. We  construct an example, which demonstrates that
the comparison theorem need not hold whenever the diffusion term
contains a delay function although the jump-diffusion coefficient
could contain a delay function just as Example \ref{jump} below
shows. Moreover, another example, Example \ref{example}, is
established to show that the comparison theorem is not necessary to
be true provided that the jump-diffusion term is non-increasing with
respect to the delay variable.

The organization of this paper goes as follows: In Section $2$ we
establish a comparison theorem for two one-dimensional
 SDDEs with pure jumps,  and similar comparison
results are given  for  SDDEs with compensator jump processes in
Section $3$.

\section{Comparison Theorem for  SDDEs with Pure Jumps}

Let $W(t), t\geq0$, be a real-valued Wiener process defined on a
certain probability space $(\Omega, \mathcal {F}, \mathbb{P})$
equipped with a filtration ${\{\mathcal {F}_t}\}_{t\geq0}$
satisfying the usual conditions (i.e., it is right continuous and
$\mathcal {F}_0$ contains all $\mathbb{P}$-null sets), and
$N(\cdot,\cdot)$ is a Poisson counting process with characteristic
measure $\lambda$ on measurable subset $\mathbb{Y}$ of $[0,\infty)$
with $\lambda(\mathbb{Y})<\infty$,
$\tilde{N}(dt,du):=N(dt,du)-\lambda(du)dt$ is a compensator
martingale process. Let $\tau>0$ and denote
$D([-\tau,0];\mathbb{R})$ the space of all c\`{a}dl\`{a}g paths from
$[-\tau,0]$ into $\mathbb{R}$ with the norm $\|u\|:=\sup_{-\tau\leq
\theta\leq 0}|u(\theta)|$. Throughout this paper, we assume that
$W(t)$ and $N(dt,du)$ are independent.

Fix $T>0$ and consider SDDE with jumps for $t\in[0,T]$
\begin{equation}\label{eq37}
\begin{split}
dX(t)&=f(X(t),X(t-\tau),t)dt+g(X(t),X(t-\tau),t)dW(t)\\
&+\int_{\mathbb{Y}}\gamma(X(t),X(t-\tau),t)\tilde{N}(dt,du)
\end{split}
\end{equation}
with initial condition $X(\theta)=\xi(\theta)\in
D([-\tau,0];\mathbb{R})$. Assume that there exist positive constants
$L_n$ such that
\begin{equation}\label{eq35}
\begin{split}
&|f(x_1,y_1,t)-f(x_2,y_2,t)|^2+|g(x_1,y_1,t)-g(x_2,y_2,t)|^2\\
&+\int_{\mathbb{Y}}|\gamma(x_1,y_1,t)-\gamma(x_2,y_2,t)|^2\lambda(du)\leq
L_n(|x_1-x_2|^2+|y_1-y_2|^2)
\end{split}
\end{equation}
for any $x_1,x_2,y_1,y_2\in\mathbb{R}$ with
$|x_1|\vee|x_2|\vee|y_1|\vee|y_2|\leq n$ and there exists a constant
$L>0$ such that for any $x,y\in\mathbb{R}$
\begin{equation}\label{eq36}
|f(x,y,t)|^2+|g(x,y,t)|^2+\int_{\mathbb{Y}}|\gamma(x,y,t,u)|^2\lambda(du)\leq
L(1+|x|^2+|y|^2).
\end{equation}

By the standard Banach fixed point theorem and truncation approach,
the following existence and uniqueness result can be found.
\begin{lemma}\label{existence}
Under conditions \eqref{eq35} and \eqref{eq36}, for initial
condition $\mathbb{E}\|\xi\|^2<\infty$, Eq. \eqref{eq37} has a
unique solution $X(t),t\in[0,T]$, with property
$\mathbb{E}\sup_{-\tau\leq t\leq T}|X(t)|^2<\infty$.
\end{lemma}

In order to state our main results, we need the following Lemma.

\begin{lemma}\label{lemma1}
Consider two one-dimensional SDEs with jumps for any $t\in[0,T]$
\begin{equation}\label{eq38}
\begin{split}
X_1(t)&=x_1+\int_0^tf_1(X_1(s),s)ds+\int_0^tg(X_1(s),s)dW(s)\\
&+\int_0^t\int_{\mathbb{Y}}\gamma_1(X_1(s^-),s,u)N(ds,du)
\end{split}
\end{equation}
and
\begin{equation}\label{eq39}
\begin{split}
X_2(t)&=x_2+\int_0^tf_2(X_2(s),s)ds+\int_0^tg(X_2(s),s)dW(s)\\
&+\int_0^t\int_{\mathbb{Y}}\gamma_2(X_2(s^-),s,u)N(ds,du).
\end{split}
\end{equation}
Assume that there exists a constant $L>0$ such that for any
$x,y\in\mathbb{R}$ and $t\in[0,T]$
\begin{equation}\label{eq14}
|f_i(x,t)-f_i(y,t)|^2+|g(x,t)-g(y,t)|^2+\int_{\mathbb{Y}}|\gamma_i(x,t,u)-\gamma_i(y,t,u)|^2\lambda(du)\leq
L|x-y|^2
\end{equation}
for $i=1,2$ with $\mathbb{E}\sup_{0\leq t\leq
T}\left(|f_i(0,t)|^2+|g(0,t)|^2+\int_{\mathbb{Y}}|\gamma_i(0,t,u)|^2\lambda(du)\right)<\infty$
and
\begin{equation}\label{eq13}
 f_1(x,t)\geq f_2(x,t) \mbox{ and } \gamma_1(x,t,u)\geq\gamma_2(x,t,u), t\in[0,T],
u\in\mathbb{Y}.
\end{equation}
Moreover assume that for any $x,y\in\mathbb{R}$ and $u\in\mathbb{Y}$
\begin{equation}\label{eq15}
x+\gamma_1(x,t,u)\leq y+\gamma_1(y,t,u) \mbox{ whenever } x\leq y.
\end{equation}
Then we have
\begin{equation}\label{eq40}
X_1(t)\geq X_2(t), \forall t\in[0,T], \mbox{ a.s. provided that }
x_1\geq x_2.
\end{equation}
\end{lemma}

\begin{proof}
By Lemma \ref{existence}, both Eq. \eqref{eq38} and Eq. \eqref{eq39}
have unique solutions, respectively. Applying the Tanaka-type
formula \cite[Theorem 152, p120]{st05}, we have for any $t\in[0,T]$
\begin{equation*}
\begin{split}
(X_2(t)-X_1(t))^+&=(x_2-x_1)^++\int_0^tI_A[f_2(X_2(s),s)-f_1(X_1(s),s)]ds\\
&+\int_0^tI_A[g(X_2(s),s)-g(X_1(s),s)]dW(s)\\
&+\int_0^t\int_{\mathbb{Y}}[(X_2(s^-)-X_1(s^-)+\gamma_2(X_2(s^-),s,u)-\gamma_1(X_1(s^-),s,u))^+\\
&-(X_2(s^-)-X_1(s^-))^+]N(ds,du)\\
&\leq\int_0^tI_A[(f_1(X_2(s),s)-f_1(X_1(s),s))+(f_2(X_2(s),s)-f_1(X_2(s),s))]ds\\
&+\int_0^tI_A[g(X_2(s),s)-g(X_1(s),s)]dW(s)\\
&+\int_0^t\int_{\mathbb{Y}}I_A(\gamma_1(X_2(s^-),s,u)-\gamma_1(X_1(s^-),s,u))N(ds,du)\\
&+\int_0^t\int_{\mathbb{Y}}[(X_2(s^-)-X_1(s^-)+\gamma_1(X_2(s^-),s,u)-\gamma_1(X_1(s^-),s,u)\\
&+\gamma_2(X_2(s^-),s,u)-\gamma_1(X_2(s^-),s,u))^+
-(X_2(s^-)-X_1(s^-))^+\\
&-I_A(\gamma_1(X_2(s^-),s,u)-\gamma_1(X_1(s^-),s,u))]N(ds,du),
\end{split}
\end{equation*}
in which $A:=\{X_2(s)-X_1(s)>0\}$ and the second inequality is due
to $x_1\geq x_2$. Noting by \eqref{eq13} that
\begin{equation*}
f_2(X_2(s),s)-f_1(X_2(s),s))\leq0 \mbox{ and }
\gamma_2(X_2(s^-),s,u)-\gamma_1(X_2(s^-),s,u)\leq0
\end{equation*}
and taking expectations, we obtain
\begin{equation*}
\begin{split}
\mathbb{E}(X_2(t)-X_1(t))^+
&\leq\mathbb{E}\int_0^tI_A[f_1(X_2(s),s)-f_1(X_1(s),s)]ds\\
&+\mathbb{E}\int_0^t\int_{\mathbb{Y}}I_A(\gamma_1(X_2(s^-),s,u)-\gamma_1(X_1(s^-),s,u))\lambda(du)ds\\
&+\mathbb{E}\int_0^t\int_{\mathbb{Y}}[(X_2(s^-)-X_1(s^-)+\gamma_1(X_2(s^-),s,u)-\gamma_1(X_1(s^-),s,u))^+
\\&-(X_2(s^-)-X_1(s^-))^+
-I_A(\gamma_1(X_2(s^-),s,u)-\gamma_1(X_1(s^-),s,u))]N(ds,du).
\end{split}
\end{equation*}
On the other hand, thanks to \eqref{eq15}, it follows that
\begin{equation*}
\begin{split}
&\mathbb{E}\int_0^t\int_{\mathbb{Y}}[(X_2(s^-)-X_1(s^-)+\gamma_1(X_2(s^-),s,u)-\gamma_1(X_1(s^-),s,u))^+
-(X_2(s^-)-X_1(s^-))^+\\
&-I_A(\gamma_1(X_2(s^-),s,u)-\gamma_1(X_1(s^-),s,u))]N(ds,du)\leq0.
\end{split}
\end{equation*}
Hence, taking into account \eqref{eq14}
\begin{equation*}
\begin{split}
\mathbb{E}(X_2(t)-X_1(t))^+&\leq
(1+\lambda^{\frac{1}{2}}(\mathbb{Y}))L^{\frac{1}{2}}\mathbb{E}\int_0^tI_{\{X_2(s)-X_1(s)>0\}}|X_2(s)-X_1(s)|ds\\
&=(1+\lambda^{\frac{1}{2}}(\mathbb{Y}))L^{\frac{1}{2}}\mathbb{E}\int_0^t(X_2(s)-X_1(s))^+ds.
\end{split}
\end{equation*}
This, in addition to Gonwall's inequality, implies
$\mathbb{E}(X_2(t)-X_1(t))^+=0$ and then yields $X_2(t)\leq X_1(t),
t\in[0,T]$, a.s. due to the fact that $(X_2(t)-X_1(t))^+$ is a
nonnegative random variable for fixed $t$, as required.
\end{proof}

\begin{remark}
{\rm  Peng and Zhu \cite[Theorem 3.1]{pz06} obtain a necessary and sufficient
condition of comparison theorem for two one-dimensional SDEs
driven by compensator jump processes such that
\begin{equation}\label{eq41}
\begin{split}
X_1(t)&=x_1+\int_0^tf_1(X_1(s),s)ds+\int_0^tg_1(X_1(s),s)dW(s)\\
&+\int_0^t\int_{\mathbb{Y}}\gamma_1(X_1(s^-),s,u)\tilde{N}(ds,du)
\end{split}
\end{equation}
and
\begin{equation}\label{eq42}
\begin{split}
X_2(t)&=x_2+\int_0^tf_2(X_2(s),s)ds+\int_0^tg_2(X_2(s),s)dW(s)\\
&+\int_0^t\int_{\mathbb{Y}}\gamma_2(X_2(s^-),s,u)\tilde{N}(ds,du).
\end{split}
\end{equation}
$X_1(t)\ge X_2(t)$ if and only if
\begin{equation*}
f_1(x,t)\geq f_2(x,t), g_1(x,t)= g_2(x,t),
\gamma_1(x,t,u)=\gamma_2(x,t,u)
\end{equation*}
as well as \eqref{eq15} holds. For Eq. \eqref{eq38} and Eq.
\eqref{eq39} which are driven by pure jump processes, we consider the
comparison result in Lemma \ref{lemma1}, where it is not necessary
to impose $\gamma_1=\gamma_2$. Clearly, Eq. \eqref{eq38} and Eq.
\eqref{eq39} can be easily transformed to Eq. \eqref{eq41} and Eq.
\eqref{eq42}, respectively. However, we shall use this Lemma to  establish a comparison theorem for SDDEs driven by jump processes.}
\end{remark}

 In the work \cite{ymy08}, where
comparison theorem of one-dimensional stochastic hybrid delay
systems is studied, a very suggestive example (Example 3.3) shows
that the comparison theorem need not hold whenever the diffusion
terms contain a delay function. While for stochastic delay systems
with jumps, the following example demonstrates that the
jump-diffusion terms could contain a delay function.

\begin{example}\label{jump}
{\rm Consider the following two one-dimensional SDEs with jumps
\begin{equation}\label{eq1}
\begin{cases}
X(t)&=c+\int_0^t\int_{-\infty}^{\infty}\gamma(u)X(s-\tau)\tilde{N}(ds,du),t\in[0,T];\\
X(\theta)&=c,\theta\in[-\tau,0)
\end{cases}
\end{equation}
and
\begin{equation}\label{eq2}
\begin{cases}
Y(t)&=\int_0^t\int_{-\infty}^{\infty}\gamma(u)Y(s-\tau)\tilde{N}(ds,du),t\in[0,T];\\
Y(\theta)&=0,\theta\in[-\tau,0)
\end{cases}
\end{equation}
where $c<0$ is a constant. We further assume that
\begin{equation}\label{eq3}
\gamma(u)>0, u\in(-\infty,\infty)
\end{equation}
and
\begin{equation}\label{eq4}
 \mbox{ and }
\tau\int_{-\infty}^{\infty}\gamma(u)\lambda(du)<1.
\end{equation}
For any $t\in[0,\tau]$
\begin{equation}\label{eq5}
\begin{split}
X(t)&=c\left(1+\int_0^t\int_0^{\infty}\gamma(u)\tilde{N}(ds,du)\right)\\
&=c\left(1+\int_0^t\int_0^{\infty}\gamma(u)N(ds,du)-\int_0^t\int_0^{\infty}\gamma(u)\lambda(du)ds\right).
\end{split}
\end{equation}
By \eqref{eq3}, combining the definition of stochastic calculus with
jumps, it follows that for $t\in[0,\tau]$
\begin{equation*}
\int_0^t\int_0^{\infty}\gamma(u)N(ds,du)>0
\end{equation*}
and
\begin{equation*}
-\int_0^t\int_0^{\infty}\gamma(u)\lambda(du)ds\geq-\int_0^{\tau}\int_0^{\infty}\gamma(u)\lambda(du)ds=-\tau\int_{-\infty}^{\infty}\gamma(u)\lambda(du).
\end{equation*}
Hence, together with \eqref{eq4}, in \eqref{eq5} $X(t)<0$ while
$Y(t)\equiv0$ for $t\in[0,\tau]$. As a consequence, we could derive
the following comparison result: the solutions $X(t)$ of Eq.
\eqref{eq1} and $Y(t)$ of Eq. \eqref{eq2} obey the property for
$t\in[0,\tau]$
\begin{equation*}
X(t)\leq Y(t) \mbox{ a.s. }
\end{equation*}}
\end{example}

Motivated by \cite[Example 3.3]{ymy08} we could also establish an
example to show that, for stochastic delay systems with jumps, the
comparison theorem need not hold if the diffusion term contains a
delay function.
\begin{example}
{\rm Consider the following two one-dimensional equations
\begin{equation}\label{eq6}
\begin{cases}
X(t)&=c+\int_0^tX(s-\tau)dB(s)-\int_0^tX(s-\tau)dN(s),t\in[0,T];\\
X(\theta)&=c,\theta\in[-\tau,0)
\end{cases}
\end{equation}
and
\begin{equation}\label{eq7}
\begin{cases}
Y(t)&=\int_0^tY(s-\tau)dB(s)-\int_0^tY(s-\tau)dN(s),t\in[0,T];\\
Y(\theta)&=0,\theta\in[-\tau,0)
\end{cases}
\end{equation}
where $c<0$ is a constant, $N$ is a Poisson process and independent
of Brownian motion $B$. Clearly, for any $t\in[0,\tau]$,
$Y(t)\equiv0$ while
\begin{equation*}
X(t)=c(1+B(t)-N(t)).
\end{equation*}
Noting that $N(t)\geq0$ and the relation
\begin{equation*}
\{(t,\omega)\in[0,\tau]\times\Omega:B(t)<-1\}\subseteq\{(t,\omega)\in[0,\tau]\times\Omega:1+B(t)-N(t)<0\},
\end{equation*}
hence
\begin{equation*}
\mathbb{P}\{(t,\omega)\in[0,\tau]\times\Omega:1+B(t)-N(t)<0\}\geq\mathbb{P}\{(t,\omega)\in[0,\tau]\times\Omega:B(t)<-1\}>0,
\end{equation*}
since $B$ obeys the normal distribution. This, together with $c<0$,
yields
\begin{equation*}
\mathbb{P}\{(t,\omega)\in[0,\tau]\times\Omega:X(t,\omega)>0\}>0.
\end{equation*}}
\end{example}

Consequently, we can conclude that  comparison theorem need not hold
if the diffusion coefficient contains a delay function. What's more,
the following example shows if the jump coefficients are  not
increasing,
 the comparison theorem also need not hold.

\begin{example}\label{example}
{\rm Consider the following two one-dimensional equations
\begin{equation}\label{eq32}
\begin{cases}
X(t)&=c-2\int_0^tX(s-\tau)dN(s),t\in[0,T];\\
X(\theta)&=c,\theta\in[-\tau,0)
\end{cases}
\end{equation}
and
\begin{equation}\label{eq33}
\begin{cases}
Y(t)&=-2\int_0^tI_{\{Y(s-\tau)<0\}}Y(s-\tau)dN(s),t\in[0,T];\\
Y(\theta)&=0,\theta\in[-\tau,0),
\end{cases}
\end{equation}
where $c<0$ is a constant and $N$ is a Poisson process with
intensity $\lambda$.}
\end{example}
By Eq. \eqref{eq32} it is easy to see that for any $t\in[0,\tau]$
\begin{equation*}
X(t)=c(1 - 2N(t)).
\end{equation*}
In what follows we intend to show
\begin{equation}\label{eq34}
\mathbb{P}\{(t, \omega)\in (0,\tau]\times\Omega: X(t, \omega)>0\}>0.
\end{equation}
Indeed, noting that
\begin{equation*}
\{1-2N(t)<0\}=\{N(t)\geq1\},
\end{equation*}
we have
\begin{equation*}
\mathbb{P}\{1-2N(t)<0\}=1-e^{-\lambda t}>0 \mbox{ whenever } 0<t
<\tau,
\end{equation*}
which further gives \eqref{eq34}. Although
\begin{equation*}
-2y\leq-2yI_{\{y<0\}} \mbox{ and } c<0,
\end{equation*}
we can not deduce that
\begin{equation*}
X(t)\leq Y (t) \mbox{ a.s. }
\end{equation*}
due to \eqref{eq34} and $Y(t)\equiv0, t\in[0, T]$.

Based on the previous discussion, now we state a comparison theorem
for SDDEs driven by pure jump processes. In the proof, Lemma \ref{lemma1}
is used.

\begin{theorem}\label{comparison 1}
Consider two one-dimensional SDDEs with pure jumps for any
$t\in[0,T]$
\begin{equation}\label{eq16}
\begin{cases}
dX_1(t)&=f_1(X_1(t),X_1(t-\tau),t)dt+g(X_1(t),t)dW(t)\\
&+\int_{\mathbb{Y}}\gamma(X_1(t^-),X_1((t-\tau)^-),t,u)N(dt,du)\\
X_1(t)&=\xi_1(t),t\in[-\tau,0],
\end{cases}
\end{equation}
and
\begin{equation}\label{eq25}
\begin{cases}
dX_2(t)&=f_2(X_2(t),X_2(t-\tau),t)dt+g(X_2(t),t)dW(t)\\
&+\int_{\mathbb{Y}}\gamma(X_2(t^-),X_2((t-\tau)^-),t,u)N(dt,du)\\
X_2(t)&=\xi_2(t),t\in[-\tau,0].
\end{cases}
\end{equation}
Assume that there exists a constant $L>0$ such that for any
$x_1,x_2,y_1,y_2,x,y\in\mathbb{R}$
\begin{equation}\label{eq22}
\begin{split}
|f_i(x_1,y_1,t)-f_i(x_2,y_2,t)|^2+&\int_{\mathbb{Y}}|\gamma(x_1,y_1,t,u)-\gamma(x_2,y_2,t,u)|^2\lambda(du)\\
&\leq L(|x_1-x_2|^2+|y_1-y_2|^2)
\end{split}
\end{equation}
with $i=1,2$ and
\begin{equation}\label{eq23}
|g(x,t)-g(y,t)|^2\leq L|x-y|^2
\end{equation}
with property $\mathbb{E}\sup_{0\leq t\leq
T}\left(|f_i(0,0,t)|^2+|g(0,t)|^2+\int_{\mathbb{Y}}|\gamma(0,0,t,u)|^2\lambda(du)\right)<\infty$.
Assume further that for $x,y,z\in\mathbb{R}$
\begin{equation}\label{eq18}
f_1(x,y,t)\geq f_2(x,y,t)
\end{equation}
and
\begin{equation}\label{eq19}
x+\gamma(x,z,t,u)\leq y+\gamma(y,z,t,u) \mbox{ whenever } x\leq y.
\end{equation}
Moreover, we suppose that $f_2$ and $\gamma$ is nondecreasing with
respect to the second variable, that is, for $t\in[0,T]$ and fixed
$x\in\mathbb{R}$ and $u\in\mathbb{Y}$,
\begin{equation}\label{eq21}
f_2(x,y,t)\geq f_2(x,z,t) \mbox{ and }
\gamma(x,y,t,u)\geq\gamma(x,z,t,u) \mbox{ whenever } y\geq z.
\end{equation}
Then we have
\begin{equation*}
X_1(t)\geq X_2(t),t\in[0,T] \mbox{ a.s. provided that } \xi_1(t)\geq
\xi_2(t) \mbox{ with } t\in[-\tau,0].
\end{equation*}

\end{theorem}

\begin{proof}
Under conditions \eqref{eq22} and \eqref{eq23}, both Eq.
\eqref{eq16} and Eq. \eqref{eq25} have unique solutions
$X_1(t),t\in[0,T]$ and $X_2(t),t\in[0,T]$, respectively. Now
consider SDDE with pure jumps for any $t\in[-\tau,T]$
\begin{equation}\label{eq20}
\begin{cases}
dX_3(t)&=f_2(X_3(t),X_1(t-\tau),t)dt+g(X_3(t),t)dW(t)\\
&+\int_{\mathbb{Y}}\gamma(X_3(t^-),X_1((t-\tau)^-),t,u)N(dt,du)\\
X_3(t)&=\xi_2(t),t\in[-\tau,0].
\end{cases}
\end{equation}
Noting by \eqref{eq18} that $f_1(x,X_1(t-\tau),t)\geq
f_2(x,X_1(t-\tau),t)$, together with $\xi_1(t)\geq \xi_2(t)$ for
$t\in[-\tau,0]$, we conclude by Lemma \ref{lemma1} that $X_1(t)\geq
X_3(t), t\in[-\tau,T],$ a.s. Next consider SDDE with pure jumps
\begin{equation*}
\begin{cases}
dX_4(t)&=f_2(X_4(t),X_3(t-\tau),t)dt+g(X_3(t),t)dW(t)\\
&+\int_{\mathbb{Y}}\gamma(X_4(t^-),X_3((t-\tau)^-),t,u)N(dt,du)\\
X_4(t)&=\xi_2(t),t\in[-\tau,0],
\end{cases}
\end{equation*}
which could be rewritten as
\begin{equation*}
\begin{cases}
dX_4(t)&=[f_2(X_4(t),X_1(t-\tau),t)+(f_2(X_4(t),X_3(t-\tau),t)-f_2(X_4(t),X_1(t-\tau),t))]dt\\
&+g(X_4(t),t)dW(t)
+\int_{\mathbb{Y}}[\gamma(X_4(t^-),X_1((t-\tau)^-),t,u)\\
&+(\gamma(X_4(t^-),X_3((t-\tau)^-),t,u)-\gamma(X_4(t^-),X_1((t-\tau)^-),t,u))]N(dt,du)\\
X_3(t)&=\xi_2(t),t\in[-\tau,0].
\end{cases}
\end{equation*}
Recalling $X_1(t)\geq X_3(t), t\in[-\tau,T]$ a.s., by \eqref{eq21}
it follows that
\begin{equation*}
f_2(x,X_1(t-\tau),t)\geq f_2(x,X_3(t-\tau),t) \mbox{ and
}\gamma(x,X_1((t-\tau)^-),t,u)\geq\gamma(x,X_3((t-\tau)^-),t,u).
\end{equation*}
Again by Lemma \ref{lemma1} $X_3(t)\geq X_4(t), t\in[-\tau,T]$ a.s.
In what follows, repeating the previous procedure we can get the
sequence
\begin{equation}\label{eq26}
X_1(t)\geq X_3(t)\geq X_4(t)\geq X_5(t)\geq\cdots\geq
X_n(t)\geq\cdots \mbox{ a.s. },
\end{equation}
where $X_n(t)$ satisfies the following equation
\begin{equation}\label{eq24}
\begin{cases}
dX_n(t)&=f_2(X_n(t),X_{n-1}(t-\tau),t)dt+g(X_n(t),t)dW(t)\\
&+\int_{\mathbb{Y}}\gamma(X_n(t^-),X_{n-1}((t-\tau)^-),t,u)N(dt,du)\\
X_n(t)&=\xi_2(t),t\in[-\tau,0].
\end{cases}
\end{equation}
In what follows we intend to show that $X_n(t)$ is a Cauchy
sequence, which has a unique limit $X(t)$, and $X(t)=X_2(t),
t\in[0,T]$, giving the desired assertion. Denote by $L_{\mathcal
{F}_t}^2([0,T];\mathbb{R})$, the space of $\mathbb{R}$-valued and
$\mathcal {F}_t$-adapted stochastic processes with
$\mathbb{E}\int_0^T|\varphi(t)|^2dt<\infty$, equipped with the norm
\begin{equation*}
\|v\|_{-\beta}:=\left(\mathbb{E}\int_0^T|v(s)|^2e^{-\beta
s}ds\right)^{\frac{1}{2}},
\end{equation*}
where $\beta$ is a positive constant to be determined. Obviously,
the norm $\|v\|_{-\beta}$ is equivalent to the original one
$\|v\|:=\mathbb{E}\int_0^T|\varphi(t)|^2dt$ for $v\in L_{\mathcal
{F}_t}^2([0,T];\mathbb{R})$. For simplicity, set
$\bar{X}_n(t):=X_n(t)-X_{n-1}(t), n\geq4$. Applying It\^o's formula
we find for any $t\in[0,T]$
\begin{equation*}
\begin{split}
\mathbb{E}&(e^{-\beta
t}|\bar{X}_n(t)|^2)\\&=\mathbb{E}\int_0^t-\beta e^{-\beta
s}|\bar{X}_n(s)|^2ds+\mathbb{E}\int_0^te^{-\beta
s}[2\bar{X}_n(s)(f_2(X_n(s),X_{n-1}(s-\tau),s)\\
&-f_2(X_{n-1}(s),X_{n-2}(s-\tau),s))+|g(X_n(s),s)-g(X_{n-1}(s),s)|^2]ds\\
&+\mathbb{E}\int_0^t\int_{\mathbb{Y}}e^{-\beta
s}\Big[2\bar{X}_n(s)\Big(\gamma(X_n(s^-),X_{n-1}((s-\tau)^-),s,u)\\
&-\gamma(X_{n-1}(s^-),X_{n-2}((s-\tau)^-),s,u)\Big)\\
&+\Big|\gamma(X_n(s^-),X_{n-1}((s-\tau)^-),s,u)
-\gamma(X_{n-1}(s^-),X_{n-2}((s-\tau)^-),s,u)\Big|^2\Big]N(ds,du).
\end{split}
\end{equation*}
This, together with \eqref{eq22} and \eqref{eq23}, yields that
\begin{equation*}
\begin{split}
\mathbb{E}(e^{-\beta
t}|\bar{X}_n(t)|^2)&\leq\mathbb{E}\int_0^t-\beta e^{-\beta
s}|\bar{X}_n(s)|^2ds+\mathbb{E}\int_0^te^{-\beta
s}[2L|\bar{X}_n(s)|^2+L|\bar{X}_{n-1}(s)|^2\\
&+2L^{\frac{1}{2}}(1+(\lambda(\mathbb{Y}))^{\frac{1}{2}})|\bar{X}_n(s)|(|\bar{X}_n(s)|+|\bar{X}_{n-1}(s)|)]ds\\
&\leq(-\beta+2L+3L^{\frac{1}{2}}(1+(\lambda(\mathbb{Y}))^{\frac{1}{2}}))\mathbb{E}\int_0^te^{-\beta
s}|\bar{X}_n(s)|^2ds\\
&+(L+L^{\frac{1}{2}}(1+(\lambda(\mathbb{Y}))^{\frac{1}{2}}))\mathbb{E}\int_0^te^{-\beta
s}|\bar{X}_{n-1}(s)|^2ds.
\end{split}
\end{equation*}
Letting
\begin{equation*}
\beta=5(L+L^{\frac{1}{2}}(1+(\lambda(\mathbb{Y}))^{\frac{1}{2}})),
\end{equation*}
we then have
$-\beta+2L+3L^{\frac{1}{2}}(1+(\lambda(\mathbb{Y}))^{\frac{1}{2}})<0$
and
\begin{equation*}
(L+L^{\frac{1}{2}}(1+(\lambda(\mathbb{Y}))^{\frac{1}{2}}))/(\beta-(2L+3L^{\frac{1}{2}}(1+(\lambda(\mathbb{Y}))^{\frac{1}{2}})))=\frac{1}{2}.
\end{equation*}
Hence
\begin{equation*}
\mathbb{E}\int_0^te^{-\beta
s}|\bar{X}_n(s)|^2ds\leq\frac{1}{2}\mathbb{E}\int_0^te^{-\beta
s}|\bar{X}_{n-1}(s)|^2ds,
\end{equation*}
which implies by induction arguments that
\begin{equation*}
\mathbb{E}\int_0^te^{-\beta
s}|\bar{X}_n(s)|^2ds\leq\frac{1}{2^{n-4}}\mathbb{E}\int_0^te^{-\beta
s}|\bar{X}_4(s)|^2ds.
\end{equation*}
Therefore $\bar{X}_n(t)$ is a Cauchy sequence in $L_{\mathcal
{F}_t}^2([0,T];\mathbb{R})$ with the norm $\|\cdot\|_{-\beta}$ and
has a unique limit denoted by $X(t)\in L_{\mathcal
{F}_t}^2([0,T];\mathbb{R})$ since the space $L_{\mathcal
{F}_t}^2([0,T];\mathbb{R})$ is a complete norm space under the norm
$\|\cdot\|_{-\beta}$. Next we show $X_2(t)=X(t)$ by the uniqueness.
In fact, by \eqref{eq22}
\begin{equation*}
\begin{split}
\mathbb{E}&\int_0^Te^{-\beta
t}\left|\int_0^t[f_2(X_n(s),X_{n-1}(s-\tau),s)-f_2(X(s),X(s-\tau),s)]ds\right|^2dt\\
&\leq LT\mathbb{E}\int_0^T\int_0^te^{-\beta(
t-s)}e^{-\beta s}(|X_n(s)-X(s)|^2+|X_{n-1}(s)-X(s)|^2)dsdt\\
&\leq LT^2\mathbb{E}\int_0^Te^{-\beta s}(|X_n(s)-X(s)|^2+|X_{n-1}(s)-X(s)|^2)ds\\
 &\rightarrow0 \mbox{ as } n\rightarrow\infty,
\end{split}
\end{equation*}
and, according to It's isometry
\begin{equation*}
\begin{split}
\mathbb{E}&\int_0^Te^{-\beta
t}\left|\int_0^t\int_{\mathbb{Y}}[\gamma(X_n(s),X_{n-1}(s-\tau),s,u)-\gamma(X(s),X(s-\tau),s,u)]N(ds,du)\right|^2dt\\
&\leq C\mathbb{E}\int_0^Te^{-\beta
t}\int_0^t\int_{\mathbb{Y}}|\gamma(X_n(s),X_{n-1}(s-\tau),s,u)-\gamma(X(s),X(s-\tau),s,u)|^2\lambda(du)dsdt\\
&\leq LCT\mathbb{E}\int_0^Te^{-\beta s}(|X_n(s)-X(s)|^2+|X_{n-1}(s)-X(s)|^2)ds\\
 &\rightarrow0 \mbox{ as } n\rightarrow\infty,
\end{split}
\end{equation*}
where $C:=2(1+T\lambda(\mathbb{Y}))$ and, carrying out the previous
arguments,
\begin{equation*}
\mathbb{E}\int_0^Te^{-\beta
t}\left|\int_0^t[g(X_n(s),s)-g(X(s),s)]dW(s)\right|^2dt\rightarrow0
\mbox{ as } n\rightarrow\infty.
\end{equation*}
As a consequence, we could conclude that $X$ satisfies
\begin{equation*}
\begin{cases}
dX(t)&=f_2(X(t),X(t-\tau),t)dt+g(X(t),t)dW(t)\\
&+\int_{\mathbb{Y}}\gamma(X(t^-),X((t-\tau)^-),t,u)N(dt,du)\\
X(t)&=\xi_2(t),t\in[-\tau,0].
\end{cases}
\end{equation*}
By the uniqueness of solution of Eq. \eqref{eq25} we conclude that
$X(t)=X_2(t)$ and, recalling \eqref{eq26}, the desired assertion is
complete.
\end{proof}

\begin{remark}
{\rm \cite{ymy08} established an example to show that condition
\eqref{eq18} is vital for the comparison theorem for SDDEs. By
Example \ref{example}, we could conclude that, if the jump diffusion
$\gamma$ is nonincreasing in second variable, namely, delay term,
the comparison theorem might not be available. Therefore the
condition \eqref{eq21} is natural. With respect to \eqref{eq19}, we
can refer to Situ \cite{st05} and Peng and Zhu \cite{pz06} for more
details for SDEs with jumps.}
\end{remark}

\begin{remark}
{\rm  By carrying out the technique of stopping times, the derived
comparison theorem can be generalized to the case where Lipschitz
condition is replaced by the Carath\'{e}odory-type condition
\cite{st05}.

}
\end{remark}

\section{Comparison Theorem for SDDEs with Compensator Jump Processes }
In the last section we establish the comparison theorem for
 SDDEs with pure jump processes. To make the content
more comprehensive, in this part we aim to discuss the comparison
problems for SDDEs with compensator jump process.

Consider two one-dimensional SDDEs with jumps for any $t\in[0,T]$
\begin{equation}\label{eq27}
\begin{cases}
dX_1(t)&=f_1(X_1(t),X_1(t-\tau),t)dt+g(X_1(t),t)dW(t)\\
&+\int_{\mathbb{Y}}\gamma(X_1(t^-),X_1((t-\tau)^-),t,u)\tilde{N}(dt,du)\\
X_1(t)&=\xi_1(t),t\in[-\tau,0],
\end{cases}
\end{equation}
and
\begin{equation}\label{eq28}
\begin{cases}
dX_2(t)&=f_2(X_2(t),X_2(t-\tau),t)dt+g(X_2(t),t)dW(t)\\
&+\int_{\mathbb{Y}}\gamma(X_2(t^-),X_2((t-\tau)^-),t,u)\tilde{N}(dt,du)\\
X_2(t)&=\xi_2(t),t\in[-\tau,0].
\end{cases}
\end{equation}
Noting that $\tilde{N}(dt,du)=N(dt,du)-\lambda(du)dt$, Eq.
\eqref{eq27} and Eq. \eqref{eq28} are equivalent to
\begin{equation}\label{eq29}
\begin{cases}
dX_1(t)&=\Big[f_1(X_1(t),X_1(t-\tau),t)-\int_{\mathbb{Y}}\gamma(X_1(t^-),X_1((t-\tau)^-),t,u)\lambda(du)\Big]dt\\
&+g(X_1(t),t)dW(t)+\int_{\mathbb{Y}}\gamma(X_1(t^-),X_1((t-\tau)^-),t,u)N(dt,du)\\
X_1(t)&=\xi_1(t),t\in[-\tau,0],
\end{cases}
\end{equation}
and
\begin{equation}\label{eq30}
\begin{cases}
dX_2(t)&=\Big[f_2(X_2(t),X_2(t-\tau),t)-\int_{\mathbb{Y}}\gamma(X_2(t^-),X_2((t-\tau)^-),t,u)\lambda(du)\Big]dt\\
&+g(X_2(t),t)dW(t)+\int_{\mathbb{Y}}\gamma(X_2(t^-),X_2((t-\tau)^-),t,u)N(dt,du)\\
X_2(t)&=\xi_2(t),t\in[-\tau,0],
\end{cases}
\end{equation}
respectively.

Based on the comparison theorem, Theorem \ref{comparison 1}, we
could derive the following comparison results for stochastic delay
systems with compensator jump processes.

\begin{theorem}\label{comparison2}
Let conditions \eqref{eq22}-\eqref{eq19} hold. Moreover, we suppose
that $f_2-\gamma$ and $\gamma$ is non-decreasing with respect to the
second variable, that is, for $t\in[0,T]$ and fixed $x\in\mathbb{R}$
and $u\in\mathbb{Y}$,
\begin{equation}\label{eq30}
f_2(x,y,t)-\int_{\mathbb{Y}}\gamma(x,y,t,u)\lambda(du)\geq
f_2(x,z,t)-\int_{\mathbb{Y}}\gamma(x,z,t,u)\lambda(du)
\end{equation}
and
\begin{equation}\label{eq31}
\gamma(x,y,t,u)\geq\gamma(x,z,t,u)
\end{equation}
whenever $y\geq z$. Then we have
\begin{equation*} X_1(t)\geq
X_2(t),t\in[0,T] \mbox{ a.s. provided that } \xi_1(t)\geq \xi_2(t)
\mbox{ with } t\in[-\tau,0].
\end{equation*}
\end{theorem}

\begin{example}
Consider two one-dimensional SDDEs with jumps
\begin{equation*}
\begin{cases}
dX_1(t)&=f_1(X_1(t),X_1(t-\tau),t)dt+g(X_1(t),t)dW(t)\\
&+\int_{\mathbb{Y}}\rho(u)f_2(X_1(t),X_1(t-\tau),t)\tilde{N}(dt,du)\\
X_1(t)&=\xi_1(t),t\in[-\tau,0],
\end{cases}
\end{equation*}
and
\begin{equation*}
\begin{cases}
dX_2(t)&=f_2(X_2(t),X_2(t-\tau),t)dt+g(X_2(t),t)dW(t)\\
&+\int_{\mathbb{Y}}\rho(u)f_2(X_2(t),X_2(t-\tau),t)\tilde{N}(dt,du)\\
X_2(t)&=\xi_2(t),t\in[-\tau,0],
\end{cases}
\end{equation*}
where $f_1, f_2, g$ satisfy the conditions \eqref{eq22}-\eqref{eq18}
and $\xi_1(t)\geq \xi_2(t), t\in[-\tau,0]$.

{\rm In what follows, we further assume that $\rho>0$,
$\int_{\mathbb{Y}}\rho(u)\lambda(du)<1$ and
\begin{equation}\label{eq43}
f_2(x,y,t)\geq f_2(x,z,t) \mbox{ whenever } y\geq z.
\end{equation}
By Theorem \ref{comparison2}, to show that $X_1(t)\geq X_2(t),
t\in[-\tau,T]$, a.s.,  it is sufficient to check conditions
\eqref{eq19}, \eqref{eq30} and \eqref{eq31}. By \eqref{eq43} and
$\rho>0$, it is easy to see that conditions \eqref{eq19} and
\eqref{eq31} hold. On the other hand, recalling
$\int_{\mathbb{Y}}\rho(u)\lambda(du)<1$, we have
\begin{equation*}
f_2(x,y,t)-\int_{\mathbb{Y}}\rho(u)\lambda(du)f_2(x,y,t)=\left(1-\int_{\mathbb{Y}}\rho(u)\right)f_2(x,y,t),
\end{equation*}
and, combining \eqref{eq43}, condition \eqref{eq30} is also true.}
\end{example}

\end{document}